\documentclass[a4paper,12pt]{amsart}

\newtheorem{proposition}{Proposition}
\newtheorem{theorem}{Theorem}
\newtheorem{lemma}{Lemma}

\newtheorem{remark}{Remark}

\numberwithin{equation}{section}
\usepackage{amssymb}
\begin{document}

\title[Well-posedness for RD eqs of BZ reaction]{A well-posedness for the reaction diffusion equations of Belousov-Zhabotinsky reaction}
\author[Kondo, Novrianti, Sawada and Tsuge]{S.~Kondo{$^\ast$}, Novrianti{$^\ast$}, O.~Sawada{$^\ast$}, N.~Tsuge{$^\dagger$}}
\address{{$^\ast$} Applied Physics Course, Faculty of Engineering, Gifu University, Yanagido 1-1, Gifu, 501-1193, Japan}
\address{{$^\dagger$} Department of Mathematics Education, Faculty of Education, Gifu University, Yanagido 1-1, Gifu, 501-1193, Japan}
\email{x3912006@edu.gifu-u.ac.jp \,\,\, okihiro@gifu-u.ac.jp}
\thanks{
	N. Tsuge's research is partially supported by Grant-in-Aid for Scientific 
	Research (C) 17K05315, Japan.
}
%%%%%%%%%%%%%%%%%%%%%%%%%%%%%%%%%%%%%%%%%%%%%%%%%%%%%%%%%%
%
% Abstract
%
%%%%%%%%%%%%%%%%%%%%%%%%%%%%%%%%%%%%%%%%%%%%%%%%%%%%%%%%%%

\begin{abstract}
The time-global existence of unique smooth positive solutions to the reaction diffusion equations of the Keener-Tyson model for the Belousov-Zhabotinsky reaction in the whole space is established with bounded non-negative initial data. Deriving estimates of semigroups and time evolution operators, and applying the maximum principle, the unique existence and the positivity of solutions are ensured by construction of time-local solutions from certain successive approximation. Invariant regions and large time behavior of solutions are also discussed.
\end{abstract}

\maketitle

%%%%%%%%%%%%%%%%%%%%%%%%%%%%%%%%%%%%%%%%%%%%%%%%%%%%%%%%%%
%
% Introduction
%
%%%%%%%%%%%%%%%%%%%%%%%%%%%%%%%%%%%%%%%%%%%%%%%%%%%%%%%%%%

\section{Introduction}
We consider the following initial value problem of the reaction diffusion equations of Keener-Tyson type for Belousov-Zhabotinsky reaction in the whole space ${\mathbb R}^n$ for $n \in {\mathbb N}$:
\[
  ({\mathrm{BZ}}) \,\, \left\{
  \begin{array}{ll}
    \displaystyle \partial_t u = \Delta u + \frac{1}{\varepsilon} u(1-u) - h v \frac{u-q}{u+q} \quad & {\text{in}} \quad {\mathbb R}^n \!\times\! (0, \infty), \\
    \displaystyle \partial_t v = d \Delta v - v + u \quad & {\text{in}} \quad {\mathbb R}^n \!\times\! (0, \infty), \\
    \displaystyle u|_{t=0} = u_0, \quad v|_{t=0} = v_0 \quad & {\text{in}} \quad {\mathbb R}^n.
  \end{array}
  \right.
\]
For the derivation, see \cite{KT86, KT88}. Here, two variables $u = u(x, t)$ and $v = v(x, t)$ stand for the unknown scalar functions at $x \in {\mathbb R}^n$ and $t > 0$ who denote the concentrations in a vessel of HBrO2 and Ce4+, respectively; $u_0 = u_0(x)$ and $v_0 = v_0(x)$ are given non-negative bounded functions. We denote $\varepsilon$, $h$, $q$ and $d$ by some positive constants. In \cite{Chen00}, an example of constants is listed-up as $\varepsilon = 0.032$, $q = 2.0 \times 10^{-4}$ and $d = 0.6$. Note that $h$ stands for the excitability which governs the dynamics of a pattern formulation. In fact, a spiral pattern appears for large $h$. Besides, a ripple (concentric circles) pattern is developed for small $h$. We have used the notation of differentiation; $\partial_t := \partial/\partial t$ and $\Delta := \sum_{i=1}^n \partial_i^2$, where $\partial_i := \partial/\partial x_i$ for $i = 1, \ldots, n$.

The aim of this paper is to establish the well-posedness theory and some basic properties of solutions to (BZ), in terms of functional analysis. Although it has already been known the solvability of (BZ) in the abstract setting of $L^2$-framework by Yagi and his collaborators \cite{TY06, YOS09}, we will give a rigorous proof of the existence of time-global unique non-negative classical solutions in $L^\infty$-setting. In our framework, we may treat more various data, including the trivial solution. Thanks to this, it is possible to prove the instability of the trivial solution. Our techniques seem to be applicable for the similar situation in domains with periodic or inhomogeneous Dirichlet or Neumann boundary conditions. Furthermore, the invariant region and large time behavior of solutions are concerned. For applying the estimates of maximum principle type, we argue certain successive approximation of solutions. To obtain uniform bounds, and to ensure positivity, some estimates for semigroups and time evolution operators are derived by arguments of relatively compact perturbation from Laplacian, via smoothing properties of the heat semigroup.

Throughout this paper, the approach of functional analysis is employed. Due to the semigroup theory, the first two equations of (BZ) are formally equivalent to the integral equations:
\begin{align}
  u(t) & = e^{t \Delta} u_0 + \int^t_0 e^{(t-s) \Delta} \! \left[ \frac{u(s) \{ 1 - u(s) \}}{\varepsilon} - h v(s) \frac{u(s)-q}{u(s)+q} \right] \! ds, \label{int-u} \\
  v(t) & = e^{d t \Delta} v_0 + \int^t_0 e^{d (t-s) \Delta} \left[ -v(s)+u(s) \right] ds,  \label{int-v}
\end{align}
since $\Delta$ generates a $(C_0)$ semigroup $\{ e^{t \Delta} \}_{t \geq 0}$ in $BUC ({\mathbb R}^n)$, so-called the heat semigroup. We will give the definition of function spaces in section~3, as well as the semigroups. To show the uniqueness, this expression is useful. Once we establish the existence of solutions to the integral equations (\ref{int-u}) and (\ref{int-v}), it is easy to confirm that solutions to the integral equations satisfy (BZ) in the classical sense by the standard argument from smoothing property of the heat semigroup.

While we denote $L := d \Delta - 1$, it is easy to see that $L$ also generates a $(C_0)$ semigroup $\{ e^{t L} \}_{t \geq 0}$ in $BUC({\mathbb R}^n)$, having an exponential decay estimate in $t \to \infty$. So, the second equation of (BZ) is rewritten as
\begin{equation}\label{int-L}
  v(t) = e^{t L} v_0 + \int^t_0 e^{(t-s) L} u(s) ds.
\end{equation}
The expression (\ref{int-L}) is benefit for proving the positivity of $v$ from positivities of $v_0$ and $u$. Our main issue is to ensure the positivity of $u$.

This paper is organized as follows. In section~2, we will state the main results. Section~3 is to recall basic properties of semigroups and time evolution operators. We will give a proof of time-local solvability in section~4. A complete proof of an invariant region of solutions is obtained in section~5, as well as a priori estimates for the extension of solutions time-globally.

Throughout this paper, we denote positive constants by $C$ the value of which may differ from one occasion to another.

%%%%%%%%%%%%%%%%%%%%%%%%%%%%%%%%%%%%%%%%%%%%%%%%%%%%%%%%%%
%
% Acknowledgment
%
%%%%%%%%%%%%%%%%%%%%%%%%%%%%%%%%%%%%%%%%%%%%%%%%%%%%%%%%%%

\vspace{2mm}
\noindent
{\bf Acknowledgment.} The authors would like to express their hearty gratitude to Professor Masaharu Nagayama for attracting them this problem, and for letting them know some results on (BZ).

%%%%%%%%%%%%%%%%%%%%%%%%%%%%%%%%%%%%%%%%%%%%%%%%%%%%%%%%%%
%
% Main Results
%
%%%%%%%%%%%%%%%%%%%%%%%%%%%%%%%%%%%%%%%%%%%%%%%%%%%%%%%%%%

\section{Main Results}
This section is devoted to state the main results of this note.

\begin{theorem}\label{th}
Let $n \in {\mathbb N}$, $\varepsilon, h, d > 0$, and let $q \in (0, 1)$. Put $\bar u \in (q, 1)$ is a root of $g(u) := u (1 - u) (u + q) - \varepsilon h q (u - q) = 0$, and $S := (q, \bar u)^2$. Let $u_0, v_0 \in BUC({\mathbb R}^n)$.

\noindent ${\mathrm{(i)}}$ \, If $u_0(x) \geq 0$ and $v_0(x) \geq 0$ for $x \in {\mathbb R}^n$, then there exists a pair $(u, v)$ of time-global unique non-negative classical solutions to ${\mathrm{(BZ)}}$ in $C([0, \infty); BUC ({\mathbb R}^n))$.

\noindent ${\mathrm{(ii)}}$ \, If $(u_0(x), v_0(x)) \in S$ for $x \in {\mathbb R}^n$, then $(u(x, t), v(x, t)) \in S$ for $x \in {\mathbb R}^n$ and $t > 0$.

\noindent ${\mathrm{(iii)}}$ \, If $u(x, t_\ast) \geq c_\ast$ and $v(x, t_\ast) \geq c_\ast$ for $x \in {\mathbb R}^n$ with some $t_\ast \geq 0$ and $c_\ast > 0$, then there exists a $T_\sharp \geq t_\ast$ such that $(u(x, t), v(x, t)) \in S$ for $x \in {\mathbb R}^n$ and $t \geq T_\sharp$.
\end{theorem}

Let us introduce the notion of an invariant region. A set $\Omega \subset {\mathbb R}^2$ is called an invariant region, if a pair $(u, v)$ of solutions to (BZ) always remains in $\Omega$. Theorem~\ref{th}-(i) implies that $[0, \infty)^2$ is an invariant region. Furthermore, the assertion (ii) tells us that the square domain $S := (q, \bar u)^2$ is an invariant region. It will be seen that $[0, m]^2$ for $m \geq 1$ are also invariant regions in Proposition~\ref{wp} in section~5.

We easily notice that there are two non-negative steady states (solutions independent of $x$ and $t$): the trivial solution $(0, 0)$ and a non-trivial one $(\widetilde u, \widetilde u)$, where $\widetilde u$ is a positive root of $\widetilde g(u) := (1-u)(u+q) - \varepsilon h(u-q) = 0$. Note that $(\widetilde u, \widetilde u) \in S$. The reader may find the linear stability or instability theories around $(\widetilde u, \widetilde u)$ in \cite{YOS09}. In addition, the assertion (iii) leads us to give large time behaviors of solutions. In fact, some global attractors are in $S$. Moreover, the trivial solution is clearly instable, which follows from the strong maximum principle.

For proving the existence theory, one can release the condition of uniform continuity for initial data. Indeed, for $u_0, v_0 \in L^\infty ({\mathbb R}^n)$, there exists a pair of time-global unique smooth non-negative solutions to (BZ). However, in this case, there is a lack of the continuity of solutions in time at $t=0$. So, the solutions belong to $C_w((0, \infty); L^\infty ({\mathbb R}^n))$, i.e., $C([\delta, \infty); L^\infty ({\mathbb R}^n))$ for $\delta > 0$.

For proving Theorem~\ref{th}-(i), we first show the existence of time-local unique non-negative classical solutions. To construct time-local solutions, the key idea is to use the certain successive approximation; see section~4. One may easily see that the solution is smooth in $t$ and $x$. Once we obtain time-local well-posedness, it is rather easy to extend the solution time-globally, since a priori bounds are derived uniformly in time and space by the maximum principle. Global bounds of solutions follow from the behaviors of those to the corresponding ordinary differential equations of the logistic type.

%%%%%%%%%%%%%%%%%%%%%%%%%%%%%%%%%%%%%%%%%%%%%%%%%%%%%%%%%%
%
% Semigroups and Time Evolution Operators
%
%%%%%%%%%%%%%%%%%%%%%%%%%%%%%%%%%%%%%%%%%%%%%%%%%%%%%%%%%%

\section{Semigroups and Time Evolution Operators}
In this section, we recall definition of function spaces and properties of the heat semigroups as well as time evolution operators.

Let $n \in {\mathbb N}$, $1 \leq p < \infty$, and let $L^p := L^p ({\mathbb R}^n)$ be the space of all $p$-th integrable functions in ${\mathbb R}^n$ with the norm $\displaystyle \| f \|_p := \left( \int_{{\mathbb R}^n} |f(x)|^p dx \right)^{1/p}$. We often omit the notation of domain $({\mathbb R}^n)$, if no confusion occurs likely. Furthermore, we do not distinguish scalar valued functions and vector, as well as function spaces. Let $L^\infty$ be the space of all bounded functions with the norm $\| f \|_\infty := {\mathrm{ess}}.\sup_{x \in {\mathbb R}^n} | f(x) |$. Define $BUC$ as the space of all bounded uniformly continuous functions. Since $L^\infty$ is a Banach space, so is its closed subset $BUC$, as well as $C(I; BUC)$ for closed interval $I \subset {\mathbb R}$. For $k \in {\mathbb N}$, let $W^{k, \infty}$ be a set of all bounded functions whose $k$-th derivatives are also bounded.

In the whole space ${\mathbb R}^n$, for $w_0 \in L^\infty ({\mathbb R}^n)$ the heat equation
\[
  ({\mathrm{H}}) \,\, \left\{
  \begin{array}{ll}
    \displaystyle \partial_t w = \Delta w \quad & {\text{in}} \quad {\mathbb R}^n \!\times\! (0, \infty), \\
    w|_{t=0} = w_0 \quad & {\text{in}} \quad {\mathbb R}^n
  \end{array}
  \right.
\]
admits a time-global unique smooth solution
\begin{align*}
  w & := w(t) := w(x, t) := (e^{t \Delta} w_0) (x) := (H_t \ast w_0) (x) \\
  & := \int_{{\mathbb R}^n} (4 \pi t)^{-n/2} \exp(-|x-y|^2/4t) w_0(y) dy
\end{align*}
in $C_w((0, \infty); L^\infty)$, where $H_t := H_t (x) := (4 \pi t)^{-n/2} \exp(-|x|^2/4t)$ is the heat kernel. Since $\| H_t \|_1 = 1$ for $t > 0$, by Young's inequality we have $\| w (t) \|_\infty \leq \|w_0\|_\infty$ for $t > 0$. In particular, if $w_0(x) \geq c$ for all $x \in {\mathbb R}^n$ with some $c \in {\mathbb R}$, then $w(x, t) \geq c$ holds true for $x \in {\mathbb R}^n$ and $t > 0$; so-called the maximum principle. Furthermore, if additionally $w_0 \in BUC$ and $w_0 \not\equiv c$, then $w(x, t) > c$ for $x \in {\mathbb R}^n$ and $t > 0$; so-called the strong maximum principle.

We easily see that for $k \in {\mathbb N}$, there exists a positive constant $C$ such that $\| \partial_i^k e^{t \Delta} w_0 \|_\infty \leq C t^{-k/2} \| w_0 \|_\infty$ for $t > 0$ and $1 \leq i \leq n$. So, $w(t) \in C^k$ for $k \in {\mathbb N}$ and $t >0$, which implies that $w(t) \in C^\infty ({\mathbb R}^n)$ for $t >0$, and then $w \in C^\infty ({\mathbb R}^n \times (0, \infty))$.

In general, for $w_0 \in L^\infty$, there is a lack of the continuity of solutions to (H) in time at $t = 0$. Note that $e^{t \Delta} w_0 \to w_0$ in $L^\infty$ as $t \to 0$, if and only if $w_0 \in BUC$. The reader may find its proof in e.g. \cite{GIM99}. Indeed, if $w_0 \in BUC$, then the solution $w \in C([0, \infty); BUC)$.

Let us consider the following initial value problem associated with the second equation of (BZ):
\[
  ({\mathrm{P_V}}) \,\, \left\{
  \begin{array}{ll}
    \partial_t \psi = d \Delta \psi - \psi + \varphi \quad & {\text{in}} \quad {\mathbb R}^n \!\times\! (0, \infty), \\
    \psi|_{t=0} = \psi_0 \quad & {\text{in}} \quad {\mathbb R}^n.
  \end{array}
  \right.
\]
Here, $\varphi := \varphi (x, t)$ is a given bounded function. We are now position to state the time-global solvability of this problem, and derive upper and lower bounds for the solutions $\psi$.

\begin{lemma}\label{lem-v}
Let $n \in {\mathbb N}$, $d > 0$, $c \geq 0$, and let $\varphi \in L^\infty ({\mathbb R}^n \times [0, \infty))$ with $\varphi (x, t) \geq c$ for $x \in {\mathbb R}^n$ and $t > 0$. If $\psi_0 \in BUC$ with $\psi_0(x) \geq c$ for $x \in {\mathbb R}^n$, then there exists a time-global unique solution to ${\mathrm{(P_V)}}$ in $C([0, \infty); BUC)$ with $\psi(x, t) \geq c$ for $x \in {\mathbb R}^n$ and $t > 0$, enjoying
\begin{equation}\label{psi}
  \| \psi (t) \|_\infty \leq \| \psi_0 \|_\infty + t  \max_{0 \leq \tau \leq t} \| \varphi (\tau) \|_\infty \quad \text{for} \,\,\, t > 0.
\end{equation}
\end{lemma}

\begin{proof}
Let $L := d \Delta - 1$. One may see that $L$ generates a $(C_0)$ semigroup $\{ e^{tL} \}_{t \geq 0}$ in $BUC$ with
\[
  \| e^{tL} \|_{{\mathcal L} (L^\infty)} := \| e^{tL} \|_{L^\infty \to L^\infty} := \sup_{\psi_0 \in L^\infty, \not\equiv 0} \frac{\| e^{tL} \psi_0 \|_\infty}{\| \psi_0 \|_\infty} \leq e^{-t} \quad {\text{for}} \,\,\, t > 0,
\]
since $e^{tL} = e^{-t} e^{d t \Delta}$. So, for $\psi_0 \in BUC$, $({\mathrm{P_V}})$ is written as
\begin{equation}\label{psi-int}
  \psi(t) = e^{tL} \psi_0 + \int_0^t e^{(t-s)L} \varphi (s) ds.
\end{equation}
The existence of a time-global unique solution follows from this formula. Taking $L^\infty$-norm into (\ref{psi-int}) above, the upper bound estimate (\ref{psi}) is easily obtained.

We next show the lower bound. If $\varphi \equiv \psi_0 \equiv c$, then $\psi \equiv c$ is a unique solution to $({\mathrm{P_V}})$. So, by (\ref{psi-int}), $\chi := \psi - c$ satisfies
\[
  \chi (t) = e^{t L} (\psi_0 - c) + \int_0^t e^{(t-s) L} \left\{ \varphi (s) - c \right\} ds \geq 0
\]
for $x \in {\mathbb R}^n$ and $t > t_\flat$. Thus, $\psi \geq c$.
\end{proof}

\begin{remark}{\rm
(i)~If $\varphi$ has some regularity, e.g. $\varphi \in L^\infty ([0, \infty); W^{1, \infty})$, then $\psi$ becomes a classical solution; $C^1$ in $t$ and $C^2$ in $x$; see the proof of Lemma~\ref{lem-c} in below. Moreover, if $\varphi$ is smooth in $t$ and $x$, then the solution $\psi$ is also smooth in $t$ and $x$.

\noindent (ii)~If either $\varphi(x, t) > c$ in some open set around $x_\flat \in {\mathbb R}^n$ and $t_\flat \in [0, \infty)$ or $\psi_0 \not\equiv c$, then $\psi(x, t) > 0$ for $x \in {\mathbb R}^n$ and $t > t_\flat$ by the strong maximum principle.
}\end{remark}

In what follows, we recall some theories and estimates for time evolution operators. Let us consider the following autonomous problem:
\[
  ({\mathrm{P_A}}) \,\, \left\{
  \begin{array}{ll}
    \partial_t \xi = \Delta \xi - \eta(x, t) \xi \quad & {\text{in}} \quad {\mathbb R}^n \!\times\! (0, \infty), \\
    \xi|_{t=0} = \xi_0 \quad & {\text{in}} \quad {\mathbb R}^n.
  \end{array}
  \right.
\]
Here, $\eta := \eta (x, t)$ is a given bounded function. We now establish the time-local solvability of $({\mathrm{P_A}})$ with upper bounds of $\xi(t)$.

\begin{lemma}\label{lem-a}
Let $n \in {\mathbb N}$, $a > 0$. Assume that $\eta \in L^\infty ({\mathbb R}^n \times [0, \infty))$ with $|\eta (x, t)| \leq a$ for $x \in {\mathbb R}^n$ and $t > 0$. If $\xi_0 \in BUC$, then there exist a $T_\ast > 0$ and a time-local unique solution to $({\mathrm{P_A}})$ in $C([0, T_\ast]; BUC)$, having $\displaystyle \| \xi (t) \|_\infty \leq \frac{4}{3} \| \xi_0 \|_\infty$ holds for $t \in [0, T_\ast]$.
\end{lemma}

\begin{proof}
The proof is based on the standard iteration. Set $\xi_1(t) := e^{t \Delta} \xi_0$,
\[
  \xi_{\ell+1}(t) := e^{t \Delta} \xi_0 - \int_0^t e^{(t-s) \Delta} \eta(s) \xi_\ell (s) ds
\]
for $\ell \in {\mathbb N}$, successively. Obviously, $\| \xi_1 (t) \|_\infty \leq \| \xi_0 \|_\infty$ for $t > 0$. Taking $\| \cdot \|_\infty$ into above, $\displaystyle \| \xi_{\ell+1} (t) \|_\infty \leq \frac{4}{3} \| \xi_0 \|_\infty$ holds for $\ell \in {\mathbb N}$, at least when $t \in [0, 1/4a]$. So, we also see that $\big\{ \xi_\ell \big\}_{\ell = 1}^\infty$ is a Cauchy sequence in $C ([0, T_\ast]; BUC)$ with some $T_\ast \geq 1/4a$. One can easily check that $\xi = \lim_{\ell \to \infty} \xi_\ell$ is a solution to $({\mathrm{P_A}})$. The uniqueness follows from the Gronwall inequality, as usual.
\end{proof}

Let $A := A(x, t) := \Delta - \eta (x, t)$, the solution to $({\mathrm{P_A}})$ can be rewritten as $\xi (t) = U (t, 0) \xi_0$, using time evolution operators $\big\{ U (t, s) \big\}_{0 \leq s \leq t}$ associated with $A$; see e.g. the book of Tanabe \cite{Tanabe79}. The upper bound above imples $\| U (t, 0) \|_{L^\infty \to L^\infty} \leq 4/3$ for $0 \leq t \leq 1/4a$, as well as $\| U (t, s) \|_{L^\infty \to L^\infty} \leq 4/3$ for $0 \leq s \leq t \leq 1/4a$.

We shall discuss a classical solution to $({\mathrm{P_A}})$. Let $\nabla := (\partial_1, \ldots, \partial_n)$.

\begin{lemma}\label{lem-c}
Adding the assumption in Lemma~$\ref{lem-a}$, suppose $t^{1/2} \nabla \eta (t) \in L^\infty ({\mathbb R}^n \times [0,\infty))$. Thus, $\xi$ is a classical solution to $({\mathrm{P_A}})$.
\end{lemma}

\begin{proof}
Although the argument is rather standard, we give a proof. It is easy to see that $\|\nabla \xi (t) \|_\infty \leq C t^{-1/2}$ for $t \in [0, T_\ast]$ by
\[
  \xi (t) = e^{t \Delta} \xi_0 - \int_0^t e^{(t-s) \Delta} \eta(s) \xi (s) ds,
\]
taking $\nabla$ and $\|\cdot\|_\infty$ into above. So, the key is to derive estimates for the second spatial derivatives. One easily has
\begin{align*}
  \| \nabla^2 \xi (t) \|_\infty & \leq \| \nabla^2 e^{t \Delta} \xi_0 \|_\infty + \int_0^t \| \nabla^2 e^{(t-s) \Delta} \eta(s) \xi (s) \|_\infty ds \\
  & \leq t^{-1} \| \xi_0 \|_\infty + \int_0^t (t-s)^{-1/2} \| \nabla \left\{ \eta (s) \xi (s) \right\} \|_\infty ds \\
  & \leq t^{-1} \| \xi_0 \|_\infty + \int_0^t (t-s)^{-1/2} C s^{-1/2} ds
  \leq C t^{-1}
\end{align*}
for $t \in (0, T_\ast']$ with some $T_\ast' \leq \min \{ 1, T_\ast \}$ and constant $C$ depending only on $n$, $\| \xi_0 \|_\infty$, $\sup_{0 \leq \tau \leq T_\ast'} \| \eta (\tau) \|_\infty$ and $\sup_{0 \leq \tau \leq T_\ast'} \tau^{1/2} \| \nabla \eta (\tau) \|_\infty$. The estimate for $\partial_t \xi$ can also be derived, similarly. By uniqueness, $\xi$ becomes a classical solution as long as it exists, at least up to $T_\ast$.
\end{proof}

In here, a kind of linearized problem of the first equation of (BZ) with a non-autonomous term is considered.
\[
  ({\mathrm{P_N}}) \,\, \left\{
  \begin{array}{ll}
    \partial_t \xi = \Delta \xi - \eta(x, t) \{ \xi - c \} + \zeta(x, t) \quad & {\text{in}} \quad {\mathbb R}^n \!\times\! (0, \infty), \\
    \xi|_{t=0} = \xi_0 \quad & {\text{in}} \quad {\mathbb R}^n.
  \end{array}
  \right.
\]
Here, $\zeta := \zeta (x, t)$ is a given bounded function; $c \geq 0$ is a constant.

\begin{lemma}\label{lem-n}
Let $n \in {\mathbb N}$, $a$, $b > 0$ and $c \geq 0$. Assume that $\eta$, $\zeta \in L^\infty ({\mathbb R}^n \times [0, \infty))$ satisfying $t^{1/2} \nabla \eta$ and $t^{1/2} \nabla \zeta$ are bounded, $|\eta| \leq a$ and $0 < \zeta \leq b$ for $x \in {\mathbb R}^n$ and $t > 0$. If $\xi_0 \in BUC$ with $\xi_0(x) \geq c$ for $x \in {\mathbb R}^n$, then there exist a $T_\dagger > 0$ and a time-local unique classical solution to $({\mathrm{P_N}})$ in $C([0, T_\dagger]; BUC)$ with $\xi (x, t) > c$ for $x \in {\mathbb R}^n$ and $t \in [0, T_\dagger]$, having $\| \xi (t) \|_\infty \leq 2 \| \xi_0 \|_\infty$ holds for $t \in [0, T_\dagger]$.
\end{lemma}

\begin{proof}
Let $\theta := \xi - c$ and $\theta_0 := \xi_0 - c \geq 0$. So, $\theta$ satisfies
\[
    \partial_t \theta = \Delta \theta - \eta(x, t) \theta + \zeta(x, t), \qquad \theta|_{t=0} = \theta_0,
\]
which is also rewritten as
\begin{equation}\label{int-n}
  \theta(t) = U (t, 0) \theta_0 + \int_0^t U (t, s) \zeta (s) ds.
\end{equation}
By Lemma~\ref{lem-a} and \ref{lem-c}, we can show the existence of a time-local unique classical solution to (\ref{int-n}), having the upper bound estimate:
\begin{align*}
  \| \theta (t) \|_\infty & \leq \| U (t, 0) \theta_0 \|_\infty + \int_0^t \| U (t, s) \zeta (s) \|_\infty ds \\
  & \leq \frac{4}{3} \| \theta_0 \|_\infty + \frac{4}{3} t \max_{0 \leq \tau \leq t} \| \zeta (\tau) \|_\infty \leq 2 \|\theta_0\|_\infty
\end{align*}
for $t \in (0, T_\dagger]$ with some $T_\dagger \leq \min \{ T_\ast, 1 / 2 b \}$. Once we have $\xi(t) \geq c$, it is clear that $\| \xi(t) \|_\infty \leq 2 \| \xi_0 \|_\infty$ in $[0, T_\dagger]$.

The lower bound of solutions follows from the maximum principle for a classical solution. We suppose that there exists $(\hat x, \hat t) \in {\mathbb R} \times (0, T_\ast]$ such that $\xi (\hat x, \hat t) = c$. Without loss of generality, $\hat t$ is taken as the first time when $\xi$ touches to $c$. At $(\hat x, \hat t)$ we see that $\partial_t \xi < 0$ in the left hand side of $({\mathrm{P_N}})$, however, $\Delta \xi \geq 0$, $\zeta > 0$ and $\eta \{ \xi - c \} = 0$ in the right hand side. This contradicts to that $\xi$ is a classical solution to $({\mathrm{P_N}})$. We can apply Oleinik's technique to avoid the situation for the case $\xi (\hat x, \hat t) \to c$ as $|\hat x| \to \infty$; see \cite{IKO02} or \cite{GG99}. Note that even if $\theta_0 = \xi_0 - c \equiv 0$, then $\theta = \xi - c > 0$ by the positivity of $\zeta$. Therefore, $\xi (x, t) > c$ for $x \in {\mathbb R}^n$ and $t \in [0, T_\dagger]$.
\end{proof}

%%%%%%%%%%%%%%%%%%%%%%%%%%%%%%%%%%%%%%%%%%%%%%%%%%%%%%%%%%
%
% Time-Local Solvability
%
%%%%%%%%%%%%%%%%%%%%%%%%%%%%%%%%%%%%%%%%%%%%%%%%%%%%%%%%%%

\section{Time-Local Solvability}
We give a complete proof of the time-local solvability in this section.

\begin{proposition}\label{prop}
Let $n \in {\mathbb N}$, $\varepsilon$, $h$, $d > 0$, and let $q \in (0, 1)$. If $u_0$, $v_0 \in BUC ({\mathbb R}^n)$ with $u_0(x) \geq 0$ and $v_0(x) \geq 0$ for $x \in {\mathbb R}^n$, then there exist $T_0 > 0$ and time-local unique classical solutions $(u, v)$ to ${\mathrm{(BZ)}}$ in $C ([0, T_0]; BUC ({\mathbb R}^n))$ with $0 \leq u (x, t) \leq 2 m$ and $0 \leq v (x, t) \leq 2 m$ for $x \in {\mathbb R}^n$ and $t \in [0, T_0]$, where $m := \max \{ \|u_0\|_\infty, \|v_0\|_\infty \}$. Furthermore, $T_0 \geq C/m$ with some constant $C > 0$ independent of $m$.
\end{proposition}

\begin{proof}
We employ an iteration argument. For making the approximation sequences, we begin with
\[
  u_1 (t) := e^{t \Delta} u_0 \quad \text{and} \quad v_1 (t) := e^{d t \Delta} v_0.
\]
For $\ell \in {\mathbb N}$, we successively define
\begin{align*}
  u_{\ell + 1} (t) & := U_\ell (t, 0) u_0 + \int_0^t U_\ell (t, s) \left[ \frac{u_\ell (s) \{ 1 - u_\ell (s) \}}{\varepsilon} + \frac{h q v_\ell (s)}{u_\ell(s) + q} \right] ds, \\
  v_{\ell + 1} (t) & := e^{t L} v_0 + \int_0^t e^{(t-s) L} u_\ell (s) ds.
\end{align*}
Here, $A_\ell := \Delta - \eta_\ell$ with $\displaystyle \eta_\ell (x,t) := \frac{h v_\ell (x, t)}{u_\ell (x, t) + q}$, and $\big\{ U_\ell (t, s) \big\}_{t \geq s \geq 0}$ is the time evolution operator associated with $A_\ell$. Note that $u_{\ell + 1}$ and $v_{\ell + 1}$ formally satisfy
\begin{equation}\label{u_l}
  \partial_t u_{\ell + 1} = A_\ell u_{\ell + 1} + \zeta_\ell = \Delta u_{\ell + 1} + \frac{u_\ell \{ 1 - u_\ell \}}{\varepsilon} - h v_\ell \frac{u_{\ell + 1}- q}{u_\ell + q}
\end{equation}
with $u_{\ell + 1}|_{t=0} = u_0$ and $\displaystyle \zeta_\ell (x, t) := \frac{u_\ell (x, t) \{ 1 - u_\ell (x, t) \}}{\varepsilon} + \frac{h q v_\ell (x, t)}{u_\ell (x, t) + q}$;
\begin{equation}\label{v_l}
  \partial_t v_{\ell + 1} = L v_{\ell + 1} + u_\ell = d \Delta v_{\ell + 1} - v_{\ell + 1} + u_\ell
\end{equation}
with $v_{\ell + 1}|_{t=0} = v_0$ for non-negative functions $u_\ell$ and $v_\ell$.

In what follows, we derive estimates for $u_\ell$, $v_\ell$, $\nabla u_\ell$ and $\nabla v_\ell$. We put
\begin{align*}
  K_{1, \ell} & := K_{1, \ell} (T) := \sup_{0 \leq t \leq T} \| u_\ell (t) \|_\infty, \\
  K_{2, \ell} & := K_{2, \ell} (T) := \sup_{0 \leq t \leq T} \| v_\ell (t) \|_\infty, \\
  K_{3, \ell} & := K_{3, \ell} (T) := \sup_{0 \leq t \leq T} t^{1/2} \| \nabla u_\ell (t) \|_\infty, \\
  K_{4, \ell} & := K_{4, \ell} (T) := \sup_{0 \leq t \leq T} t^{1/2} \| \nabla v_\ell (t) \|_\infty
\end{align*}
for $T > 0$ and $\ell \in {\mathbb N}$. For deriving the uniform estimates, we will use the induction argument for $\ell$.

For $\ell = 1$, by $0 \leq u_0 \leq m$ and $0 \leq v_0 \leq m$, we easily see that $0 \leq u_1 (t) \leq \| u_0 \|_\infty$, $0 \leq v_1 (t) \leq \| v_0 \|_\infty$, $t^{1/2} \| \nabla u_1 (t) \|_\infty \leq \| u_0 \|_\infty$ and $t^{1/2} \| \nabla v_1 (t) \|_\infty \leq \| v_0 \|_\infty$ for $t > 0$ by the maximum principle. Thus,
\begin{equation}\label{k1}
  K_{j, 1} \leq m \quad \text{for} \quad T > 0 \,\,\, \text{and} \,\,\, 1 \leq j \leq 4.
\end{equation}

For $\ell = 2$, before estimating $u_2$ and $v_2$, we give bounds for $\eta_1$ and $\zeta_1$. By $u_1 \geq 0$, $v_1 \geq 0$ and (\ref{k1}), it holds that
\[
  \| \eta_1 \|_\infty \leq \frac{h}{q} m =: a_1 \quad \text{and} \quad 0 \leq \zeta_1 \leq \frac{m(1+m)}{\varepsilon} + hm := b_1.
\]
So, by Lemma~\ref{lem-a} and \ref{lem-n}, it holds that
\begin{align*}
  \| u_2 (t) \|_\infty & \leq \| U_1 (t, 0) u_0 \|_\infty + \int_0^t \| U_1 (t, s) \zeta_1(s) \|_\infty ds \\
  & \leq \frac{4}{3} \| u_0 \|_\infty + \int_0^t \frac{4}{3} b_1 ds \leq 2 m
\end{align*}
provided if $t \leq T_{\dagger, 2}$ with some $T_{\dagger, 2} > 0$. Furthermore, since $u_2$ is a classical solution to (\ref{u_l}) with $\ell = 1$ by Lemma~\ref{lem-c} with $c = 0$, we can apply the maximum principle to obtain $u_2 (x, t) \geq 0$ for $x \in {\mathbb R}^n$ and $t \in [0, T_{\dagger, 2}]$. To get the estimate for $K_{3, 2} = \sup_t t^{1/2} \| \nabla u_2 (t) \|_\infty$, we use the expression by the heat semigroup:
\[
  u_2 (t) = e^{t \Delta} u_0 + \int_0^t e^{(t-s) \Delta} \left[ - \eta_1 (s) u_2 (s) + \zeta_1 (s) \right] ds.
\]
Hence, it holds that
\[
  t^{1/2} \| \nabla u_2 (t) \|_\infty \leq \| u_0 \|_\infty + t^{1/2} \int_0^t (t-s)^{-1/2} [ a_1 \| u_2(s) \|_\infty + b_1] ds \leq 2m
\]
for $t \in (0, T_{\dagger, 2}']$ with some $T_{\dagger, 2}' \leq T_{\dagger, 2}$. On the other hand, by Lemma~\ref{lem-v} with $c = 0$, it holds that
\[
  0 \leq v_2 (x, t) \leq \| v_0 \|_\infty + t \sup_{ 0 \leq \tau \leq t} \| v_1 (\tau) \|_\infty \leq 2m
\]
for $x \in {\mathbb R}^n$ and $t \in [0, 1]$. For deriving the estimate for $\nabla v_2$, we appeal to the heat semigroup, again, to have
\[
  t^{1/2} \| \nabla v_2 (t) \| \leq \| v_0 \| + t^{1/2} \int_0^t \| \nabla e^{(t-s) \Delta} \left[ - v_2 (s) + u_1 (s) \right] \|_\infty ds \leq 2m
\]
for $t \in (0, T_{\flat, 2}]$ with $T_{\flat, 2} \leq 1$. So, let $T_2 := \min \{ T_{\dagger, 2}', T_{\flat, 2} \}$, we have
\begin{equation}\label{k2}
  u_2 \geq 0, \quad v_2 \geq 0, \quad K_{j, 2} \leq 2m \quad \text{for} \quad T \leq T_2, \quad 1 \leq j \leq 4.
\end{equation}

As the similar discussion, there exists a $T_0 \leq T_2$ such that
\begin{equation}\label{k3}
  u_3 \geq 0, \quad v_3 \geq 0, \quad K_{j, 3} \leq 2m \quad \text{for} \quad T \leq T_0, \quad  1 \leq j \leq 4.
\end{equation}
Note $T_0 \geq C/m$ with some $C > 0$. The proof is essentially the same as that for $\ell \geq 4$ in below. So, the detail is omitted in here.

Let $\ell \geq 4$. We now assume that
\begin{equation}\label{kl}
  u_\ell \geq 0, \quad v_\ell \geq 0, \quad K_{j, \ell} \leq 2m \quad \text{for} \quad T \leq T_0, \quad  1 \leq j \leq 4
\end{equation}
hold. We will compute estimates for $u_{\ell + 1}$ and $v_{\ell + 1}$. By assumption, 
\[
  \| \eta_\ell \|_\infty \leq \frac{2h}{q} m =: a \quad  \text{and} \quad 0 \leq \zeta_\ell \leq \frac{2m (1+2m)}{\varepsilon} + 2hm := b
\]
hold for $t \in [0, T_0]$. Hence, by Lemma~\ref{lem-a} and \ref{lem-n}, one can see that
\begin{align}\label{u_l+1}
  \| u_{\ell+1} (t) \|_\infty & \leq \| U_\ell (t, 0) u_0 \|_\infty + \int_0^t \| U_\ell (t, s) \zeta_\ell (s) \|_\infty ds \\
  & \leq \frac{4}{3} m + \int_0^t \frac{4}{3} b ds \leq 2 m \notag
\end{align}
for $t \in [0, T_0]$. Note that we took $T_0 \leq m/2b$ in here. Since $u_\ell$ is a classical  solution to (\ref{u_l}), we can apply the maximum principle to obtain $u_{\ell + 1} (x, t) \geq 0$ for $x \in {\mathbb R}^n$ and $t \in [0, T_0]$. For using the expression
\[
  u_{\ell + 1} (t) = e^{t \Delta} u_0 + \int_0^t e^{(t-s) \Delta} \left[ - \eta_\ell (s) u_{\ell + 1}(s) + \zeta_\ell (s) \right] ds,
\]
we take $\nabla$ and $\| \cdot \|_\infty$ into above to obtain that
\[
  t^{1/2} \| \nabla u_{\ell + 1} (t) \|_\infty \leq \| u_0 \|_\infty + t^{1/2} \!\! \int_0^t (t-s)^{-1/2} [ a \| u_{\ell + 1} (s) \|_\infty + b] ds \leq 2m
\]
for $t \in [0, T_0]$ by (\ref{u_l+1}). Besides, by Lemma~\ref{lem-v},
\[
  0 \leq v_{\ell + 1} (x, t) \leq \| v_0 \|_\infty + t \sup_{ 0 \leq \tau \leq t} \| v_\ell (\tau) \|_\infty \leq 2m
\]
holds for $x \in {\mathbb R}^n$ and $t \in [0, T_0]$. By the heat semigroup, we obtain
\[
  t^{1/2} \| \nabla v_2 (t) \| \leq \| v_0 \| + t^{1/2} \int_0^t \| \nabla e^{(t-s) \Delta} \left[ - v_2 (s) + u_1 (s) \right] \|_\infty ds \leq 2m
\]
for $t \in [0, T_0]$. Therefore, 
\[
  u_{\ell+1} \geq 0, \quad v_{\ell+1} \geq 0, \quad K_{j, \ell+1} \leq 2m \quad \text{for} \quad T \leq T_0, \quad 1 \leq j \leq 4.
\]
Thus, $(\ref{kl})$ holds true for all $\ell \in {\mathbb N}$.

One may see that $u_\ell$ and $v_\ell$ are continuous in $t \in [0, T_0]$ for $\ell \in {\mathbb N}$. It is also easy to see that $\big\{ u_\ell, v_\ell, t^{1/2} \nabla u_\ell, t^{1/2} \nabla v_\ell \big\}_{\ell = 1}^\infty$ are Cauchy sequences in $C([0, T_0]; BUC)$. We denote $(u, v, \hat u, \hat v)$ by the limit functions of $(u_\ell, v_\ell, t^{1/2} \nabla u_\ell, t^{1/2} \nabla v_\ell)$ as $\ell \to \infty$. The coincidences $\hat u = t^{1/2} \nabla u$ and $\hat v = t^{1/2} \nabla v$ hold, obviously. The uniqueness follows from the Gronwall inequality, directly. Furthermore, by construction, $0 \leq u (x, t) \leq 2m$ and $0 \leq v (x, t) \leq 2m$ for $x \in {\mathbb R}^n$ and $t \in [0, T_0]$, as well as $(u, v)$ is a pair of the time-local unique classical solutions to (BZ). This completes the proof of Proposition~\ref{prop}.
\end{proof}

\begin{remark}{\rm
(i)~For $k \in {\mathbb N}$, it is possible to construct $u(t), v(t) \in C^k ({\mathbb R}^n)$ for $t \in (0, T_k]$, if $T_k$ is chosen small enough. Nevertheless, the solution is unique as long as it exists, one can extend the existence time of the solution up to $T_0$ having bounds for $k$-th derivatives. We hence confirm that $u(t) \in C^k ({\mathbb R}^n)$ for all $k \in {\mathbb N}$ and $t \in (0, T_0]$, which means that $u(t), v(t) \in C^\infty ({\mathbb R}^n)$ in $t \in (0, T_0]$, as well as $u, v \in C^\infty ({\mathbb R}^n \times (0, T_0])$.

\noindent (ii)~This iteration procedure also works for proving $u \geq q$ and $v \geq q$, provided if $u_0 \geq q$ and $v_0 \geq q$. Since $u_\ell \geq q$ and $v_\ell \geq q$ hold for $\ell \in {\mathbb N}$ by Lemma 1 and 4 with $c = q$, as the same way as above, we ensure that the limits also satisfy $q \leq u(x, t) \leq 2m$ and $q \leq v(x, t) \leq 2m$ for $x \in {\mathbb R}^n$ and $t \in [0, T_0]$.
}\end{remark}

%%%%%%%%%%%%%%%%%%%%%%%%%%%%%%%%%%%%%%%%%%%%%%%%%%%%%%%%%%
%
% Invariant Region
%
%%%%%%%%%%%%%%%%%%%%%%%%%%%%%%%%%%%%%%%%%%%%%%%%%%%%%%%%%%

\section{Invariant Region}
In this section, invariant regions are discussed. We first show that the solutions obtained by Proposition~\ref{prop} can be extended time-globally.

\begin{proposition}\label{wp}
Let $n \in {\mathbb N}$, $\varepsilon$, $h$, $d > 0$ and $q \in (0, 1)$. If $u_0$, $v_0 \in BUC ({\mathbb R}^n)$ with $0 \leq u_0(x) \leq m$ and $0 \leq v_0(x) \leq m$ for $x \in {\mathbb R}^n$ with some $m \geq 1$, then there exists a time-global unique classical solutions $(u, v)$ to ${\mathrm{(BZ)}}$ in $C ([0, \infty); BUC ({\mathbb R}^n))$ with $0 \leq u (x, t) \leq m$ and $0 \leq v (x, t) \leq m$ for $x \in {\mathbb R}^n$ and $t > 0$.
\end{proposition}

\begin{proof}
By Proposition~\ref{prop}, we have already obtained a pair of time-local unique classical solutions $(u(x, t), v(x, t)) \in [0, 2m]^2$ for $x \in {\mathbb R}^n$ and $t \in [0, T_0]$. In what follows, we will derive the a priori estimates $u \leq m$ and $v \leq m$ for $t \in [0, T_0]$. It is enough to consider the local behavior of solutions. Using the same argument in the proof of Lemma~4, there does not exist $(\tilde x, \tilde t)$ such that $u(\tilde x, \tilde t) > m$ by $m \geq 1$. So, we have $u \leq m$. Furthermore, since $u \leq m$ and $v_0 \leq m$, one can also see that $v > m$ never happened. So, $v \leq m$.

Gathering the time-local solvability, uniqueness and upper bounds, we can extend the solution up to $t \in [0, 2 T_0]$. Repeating this argument infinitely many times, we obtain a time-global unique classical solutions $(u, v) \in [0, m]^2$.
\end{proof}

Note that Theorem~\ref{th}-(i) immediately follows from Proposition~\ref{wp}. And also, this implies that $[0, m]^2$ is an invariant region for $m \geq 1$. We are now position to show that $S := (q, \bar u)^2$ is an invariant region.

\begin{proof}[Proof of Theorem~$\ref{th}$-${\mathrm{(ii)}}$]
Let $(u_0, v_0) \in S$. By Proposition~\ref{prop}, Remark~2-(ii) and Proposition~\ref{wp}, we have obtained a time-global unique smooth solutions having the lower and upper bounds $(u, v) \in [q, 1]^2$. So, it is required to show that $u$ and $v$ never touched to $\bar u \in (q, 1)$. We assume that there exists $(\bar x, \bar t)$ such that $u (\bar x, \bar t) = \bar u$. Without loss of generality, we take $\bar t \in (0, \infty)$ is the first time, and $\bar x \in {\mathbb R}^n$. Since $u, v \geq q$ and $\bar u$ is a positive root of $g(u) = 0$, at $(\bar x, \bar t)$ we see that $\partial_t u >0$, $\Delta u \leq 0$ and $\displaystyle \frac{1}{\varepsilon}u(1-u)-hv\frac{u-q}{u+q} \leq 0$. This contradicts to that $u$ is a solution. One can avoid the case $u(x, t) \to \bar u$ at $|x| \to \infty$ by Oleinik's technique.

Similarly, if there exists $(\bar x, \bar t)$ such that $v(\bar x, \bar t) = \bar u$, then at $(\bar x, \bar t)$ we see that $\partial_t v > 0$, $d \Delta v \leq 0$ and $- v + u \leq - v + \bar u = 0$. This contradicts. It is also easy to see that $u$ and $v$ never touched to $q$, as the same arguments above. Therefore, $(u, v) \in S$.
\end{proof}

Finally, we will give the remaining parts of the proof of Theorem~\ref{th}.

\begin{proof}[Proof of Theorem~$\ref{th}$-${\mathrm{(iii)}}$]
We now put $m := \max \{ \| u_0 \|_\infty, \| v_0 \|_\infty \} > 1$ and $c_\ast \in (0, q)$, without loss of generality. Applying Lemma~\ref{lem-v}, Lemma~\ref{lem-n} with $c = c_\ast$ and Proposition~\ref{wp}, it is easy to see that $c_\ast \leq u(x, t) \leq m$ and $c_\ast \leq v(x, t) \leq m$ for $x \in {\mathbb R}^n$ and $t > t_\ast$. Let $\rho := \rho(t)$ be the solution to the following ordinary differential equation of logistic type:
\[
  \rho' = \frac{1}{\varepsilon} \rho (1 - \rho) \quad {\text{for}} \quad t > t_\ast, \quad \rho(t_\ast) = c_\ast.
\]
Note that $0 < c_\ast < q < 1$, and then $\rho$ is monotone increasing. So, there exists a $T_{\sharp1} > t_\ast$ such that $\rho (T_{\sharp1}) = q$. By the argument of the maximum principle, $u(x, t) \geq \rho(t)$ for $x \in {\mathbb R}^n$ and $t \in [t_\ast, T_{\sharp1}]$, that is to say, $\rho$ is a subsolution of $u$ up to $T_{\sharp1}$.

We secondly consider that $\sigma := \sigma (t)$ is the solution to
\[
  \sigma' = G_m (\sigma) := \frac{1}{\varepsilon} \sigma (1 - \sigma) - h m \frac{\sigma - q}{\sigma + q} \quad {\text{for}} \quad t > T_{\sharp1}, \quad \sigma (T_{\sharp1}) = q.
\]
Note that there exists $q_1 \in (q, 1)$ such that $G_m (q_1) = 0$, and $\sigma (t) \to q_1$ as $t \to \infty$. Since $\sigma$ is monotone increasing, for $q_2 \in (q, q_1)$, there exists a $T_{\sharp2} \geq T_{\sharp1}$ such that $\sigma (T_{\sharp2}) = q_2$. Again, $\sigma$ is a subsolution of $u$, we thus see that $u (x, t) \geq q_2 > q$ for $x \in {\mathbb R}^n$ and $t \geq T_{\sharp2}$.

Thirdly, we derive a lower bound of $v$. Let $\nu := \nu (t)$ be a solution to
\[
  \nu' = - \nu + q_2 \quad {\text{for}} \quad t > T_{\sharp2}, \quad \nu (T_{\sharp2}) = c_\ast.
\]
Obviously, $\nu$ is monotone increasing, and $\nu (t) \to q_2$ as $t \to \infty$. Hence, for $q_3 \in (q, q_2)$, there exists a $T_{\sharp3} \geq T_{\sharp2}$ such that $\nu (T_{\sharp3}) = q_3$. Since $\nu$ is a subsolution of $v$, we have $v (x, t) \geq q_3 > q$ for $x \in {\mathbb R}^n$ and $t \geq T_{\sharp3}$.

In what follows, we shall derive upper bounds of $u$ and $v$. Let us define $\kappa := \kappa (t)$ as the solution to
\[
  \kappa' = G_\ast (\kappa) := \frac{1}{\varepsilon} \kappa (1 - \kappa) - h q_3 \frac{\kappa - q}{\kappa + q} \quad {\text{for}} \quad t > T_{\sharp3}, \quad \kappa (T_{\sharp3}) = m.
\]
Note that $\kappa$ is monotone decreasing, and $\kappa (t) \to \kappa_\ast$ as $t \to \infty$, where $\kappa_\ast \in (q, \bar u)$ satisfies $G_\ast (\kappa) = 0$. For $u_\ast \in (\kappa_\ast, \bar u)$, there exists a $T_{\sharp4} \geq T_{\sharp3}$ such that $\kappa (T_{\sharp4}) = u_\ast$. Since $v \geq q_3$ for $t \geq T_{\sharp3}$, it holds that $u (x, t) \leq \kappa (t)$ for $x \in {\mathbb R}^n$ and $t \geq T_{\sharp3}$. That is to say, $\kappa$ is a supersolution of $u$. Moreover, $u (x, T_{\sharp4}) \leq u_\ast$ for $x \in {\mathbb R}^n$. We thus see that $u (x, t) \leq u_\ast < \bar u$ for $x \in {\mathbb R}^n$ and $t \geq T_{\sharp4}$.

Since $u \leq u_\ast < \bar u$ for $t \geq T_{\sharp4}$, there exists a $T_\sharp \geq T_{\sharp4}$ such that $v (x, T_\sharp) \leq \bar u$ for $x \in {\mathbb R}^n$, by observing a supersolution of $v$:
\[
  \mu' = - \mu + u_\ast \quad {\text{for}} \quad t > T_{\sharp4}, \quad \mu (T_{\sharp4}) = m.
\]
Hence, $v (x, t) < \bar u$ for $x \in {\mathbb R}^n$ and $t > T_\sharp$.

Note that $S$ is an invariant region by Theorem~\ref{th}-(ii). Therefore, summing up the arguments above, $(u(x, t), v(x, t)) \in S$ for $x \in {\mathbb R}^n$ and $t > T_\sharp$. This completes the proof of Theorem~\ref{th}-(iii).
\end{proof}

\begin{remark}{\rm
(i)~Looking at the proof above, we find the following fact. Let $\bar \kappa \in (q, \bar u)$ be a root of $\kappa (1 - \kappa) (\kappa + q) - \varepsilon h q_1 (\kappa - q) = 0$. For $q_\natural \in [q, q_1)$ and $u_\natural \in (\bar \kappa, \bar u]$, then there exists a $T_\natural \geq t_\ast$ such that $(u(x, t), v(x, t)) \in S_\natural := (q_\natural, u_\natural)^2 \subset S$ for $x \in {\mathbb R}^n$ and $t \geq T_\natural$. Note that $S_\natural$ is an invariant region depending on $m$.

\noindent (ii)~It is still open whether $(u, v) \in S$ for large $t$, when $u_0 \geq 0$, $v_0 \geq 0$ and either $u_0 \not\equiv 0$ or  $v_0 \not\equiv 0$.
}\end{remark}

%%%%%%%%%%%%%%%%%%%%%%%%%%%%%%%%%%%%%%%%%%%%%%%%%%%%%%%%%%
%
% Bibliography
%
%%%%%%%%%%%%%%%%%%%%%%%%%%%%%%%%%%%%%%%%%%%%%%%%%%%%%%%%%%

\end{document}